\documentclass[10pt,a4paper]{amsart}
\usepackage[latin1]{inputenc}
\usepackage{amsmath,            % paquetes para matem\'{a}ticas, preparados por la AMS
            amsfonts,
            amssymb,
            amsthm}

\author{Javier Cilleruelo}
\address{Instituto de Ciencias Matem\'aticas (CSIC-UAM-UC3M-UCM) and
Departamento de Matem\'aticas\\
Universidad Aut\'onoma de Madrid\\
28049, Madrid, Espa\~na} \email{franciscojavier.cilleruelo@uam.es}
\author{Rafael Tesoro}
\address{Departamento de Matem\'aticas\\
Universidad Aut\'onoma de Madrid\\
28049, Madrid, Espa\~na} \email{rafael.tesoro@estudiante.uam.es}
\title{Dense infinite $B_h$ sequences}
\date{2012}

\newtheorem{thm}{Theorem}[section]

\newtheorem{lem}[thm]{Lemma}
\newtheorem{prop}[thm]{Proposition}

\theoremstyle{definition}

\theoremstyle{remark}
\newtheorem{rem}[thm]{Remark}
\numberwithin{equation}{section}
\newcommand{\A}{\mathcal{A}}
\newcommand{\B}{\mathcal{B}}

\newcommand{\Z}{\mathbb{Z}}

\newcommand{\p}{{\mathfrak{p}}}
\newcommand{\q}{{\mathfrak{q}}}

\newcommand{\ud}{\,\mathrm{d}}

\usepackage{hyperref}          % para enlaces dentro del PDF

\begin{document}

\begin{abstract}
For $h=3$ and $h=4$ we prove the existence of  infinite $B_h$ sequences $\B$  with counting function
 $$\mathcal{B}(x)= x^{\sqrt{(h-1)^2+1}-(h-1) + o(1)}.$$ This result
extends a construction of I. Ruzsa for $B_2$ sequences.
\end{abstract}

\maketitle

%\tableofcontents

\section{Introduction}

Let $h \ge 2$ be an integer. We say that a sequence $\B$ of positive
integers is a $B_h$ sequence if  all the sums
\[
     b_1 + \cdots + b_h, \qquad (b_k \in \mathcal{B}, \; 1 \le k \le h),
\]
are distinct subject to $b_1 \le b_2 \le \cdots \le b_h.$ The study
of the size of finite $B_h$ sets (or the growing of the counting
function of  infinite $B_h$ sequences) is a classic topic in
combinatorial number theory. We define
$$F_h(n)=\max\{|\B|:\ \B \text{ is } B_h,\ \B\subset [1,n]\}.$$
A trivial counting argument proves that $F_h(n)\le(C_h+o(1))
n^{1/h}$ for a constant $C_h$ (see \cite{Ci} and \cite{Gr} for
non trivial upper bounds for $C_h$) and consequently that
$\mathcal B(x)\ll x^{1/h}$ when $\mathcal B$ is an infinite $B_h$
sequence.

There are three algebraic constructions (\cite{BC}, \cite{S} and
\cite{GT}) of finite $B_h$ sets showing that $F_h(n)\ge
n^{1/h}(1+o(1))$. It is probably true that
$F_h(n)\sim n^{1/h}$ but this is an open problem, except for the case
$h=2$ for which Erd\H os and Turan \cite{ET} did prove that
$F_2(n)\sim n^{1/2}$. It is unknown whether $\lim_{n\to
\infty} F_h(n)/n^{1/h}$ exists for $h\ge 3$. For further
information about $B_h$ sequences see \cite[\S \ II.2]{HR} or
\cite{O}.

Erd\H{o}s conjectured for all $\epsilon >0$ the existence of an
infinite $B_h$ sequence $\B$ with counting function $\B(x) \gg
x^{1/h - \epsilon}$. It is believed that $\epsilon$ cannot be
removed from the last exponent,  however this has only been  proved
for $h$ even. On the other hand, the \emph{greedy} algorithm produces an infinite $B_h$ sequence
$\mathcal B$ with
\begin{equation} \label{eq:greedy-density}
    \mathcal{B}(x) \gg x^{\frac 1{2h-1}} \quad (h \ge 2).
\end{equation}
Until now the exponent $1/(2h-1)$ has been the largest known for the growth of a $B_h$ sequence when $h\ge 3$. For the case $h=2$,
Atjai, Koml\'{o}s and Szemer\'{e}di  \cite{AKS} proved that there exists a
$B_2$ sequence (also called a Sidon sequence)  with $\B(x)\gg (x\log
x)^{1/3}$, improving by a power of logarithm the lower bound
\eqref{eq:greedy-density}. So far the highest improvement of
\eqref{eq:greedy-density} for the case $h=2$  was achieved by Ruzsa
(\cite{R}). He constructed, in a clever way, an infinite Sidon
sequence $\mathcal B$ satisfying
\[
    \mathcal B(x)=x^{\sqrt{2} - 1 + o(1)}. %, \qquad \gamma = \sqrt{2} - 1.
\]

 Our aim is to adapt  Ruzsa's ideas to build  dense infinite $B_3$ and $B_4$ sequences and so improve the
lower bound \eqref{eq:greedy-density} for  $h = 3$ and $h=4$.
\begin{thm} \label{thm:our_thm} For $h=2,3,4$
    there is an infinite $B_h$ sequence $\mathcal B$ with counting function
    \[
        \mathcal{B}(x) = x^{\sqrt{(h-1)^2+1}-(h-1) + o(1)}.
    \]
\end{thm}

 The starting point in
Ruzsa's construction were the numbers $\log p$, $p$ prime, which
form an infinite Sidon set of \emph{real}
 numbers. Instead  we part from  the arguments of the Gaussian
 primes, which also have the same $B_h$ property with the additional advantage of being a bounded sequence.
This idea was suggested in \cite{CR} to simplify  the original
construction of Ruzsa and was written in detail for $B_2$ sequences
in \cite{M}.

We believe that the theorem can be extended to all $h$.
Indeed, except for Lemma 3.3 the proof presented here works for all $h \ge 2$. 
The technical details in Lemma 3.3 become significantly involved as $h$ increases and we are still looking for a proof for the general case. 
%Indeed we have written all the proof for all $h\ge 2$ except for Lemma \ref{inte} where we have only considered the cases $h=2,3$ and $4$ 
%because the technical details .

\section{The Gaussian arguments}

\par For each rational prime $p\equiv 1\pmod 4$ we consider the Gaussian prime $\p$ of $\Z[i]$ such that
\begin{equation*}
\p := a+bi, \qquad p=a^2+b^2, \quad a > b > 0,
\end{equation*}
so  the argument of $ \p$ defined by $\p=\sqrt p\ e^{2 \pi i \, \theta(\p)}$
is a real number in the interval $(0,1/8)$. We
will use several times through the paper the following lemma that can be seen as a
measure of the quality of the $B_h$ property of this sequence of
real numbers.
\begin{lem}\label{lem:desi} Let
$\p_1,\cdots, \p_h,\p'_1,\cdots ,\p'_h$ be distinct Gaussian primes
satisfying \\
$0<\theta(\p_r),\theta(\p'_r)<1/8,\ r=1,\cdots ,h$. The
following inequality holds:
$$
\left |\sum_{r=1}^h\left (\theta(\p_r)-\theta(\p'_r)\right)\right |>\frac 1{7 \, |\p_1\cdots \p_h\ \p'_1 \cdots \p'_h | }.
$$
\end{lem}
\begin{proof} It is clear that
\begin{equation}\label{eq:desi-1}
\sum_{r=1}^h\left (\theta(\p_r)-\theta(\p'_r)\right) \equiv
\theta(\p_1\cdots \p_h\overline{\p'_1\cdots
\p'_h}) \pmod 1.
\end{equation}
Since $\Z[i]$ is a unique
factorization domain, all the primes are in the first octant and are
all distinct, the Gaussian integer $\p_1\cdots
\p_h\overline{\p'_1\cdots \p'_h}$ cannot be a real integer. Using this fact and the inequality $\arctan(1/x)>0.99/x$ for $x\ge\sqrt{5\cdot 13}$ we have
\begin{eqnarray} \label{eq:desi-2}
|\theta(\p_1\cdots \p_h\overline{\p'_1\cdots \p'_h})|&\ge& \|\theta(\p_1\cdots \p_h\overline{\p'_1\cdots \p'_h})\|\\ &\ge&
\frac 1{2\pi}\arctan\left (\frac 1{|\p_1\cdots \p_h\overline{\p'_1}\cdots
\overline{\p'_h}|}\right )\notag \\&>&\frac 1{7|\p_1\cdots
\p_h\overline{\p'_1}\cdots \overline{\p'_h}|}, \notag
\end{eqnarray}
where $\|\cdot  \|$ means the distance to $\Z$.
The lemma follows from  \eqref{eq:desi-1} and  \eqref{eq:desi-2}.
\end{proof}

We illustrate the $B_h$ property of the arguments of the Gaussian
primes with a quick construction, based on them, of a finite $B_h$
set which is only a $\log x$ factor below the optimal bound.

\begin{thm} The set $\A=\{\lfloor x\theta(\p)\rfloor ,\ |\p|\le
\left(\frac x{7 h} \right)^{\frac 1{2h}}\}\subset [1,x]$ is a $B_h$ set with\\
$|\A|\gg x^{1/h}/\log x$.
\end{thm}
\begin{proof} If \[ \lfloor x\theta(\p_1)\rfloor +\cdots + \lfloor
x\theta(\p_h)\rfloor= \lfloor x\theta(\p'_1)\rfloor +\cdots +
\lfloor x\theta(\p'_h)\rfloor \] then
\[ x\left |\theta(\p_1)+\cdots
+\theta(\p_h)-\theta(\p'_1)-\cdots -\theta(\p'_h)\right |\le h.\]
If the Gaussian primes are distinct then Lemma \ref{lem:desi} implies that
\begin{equation*}|\theta(\p_1)+\cdots
+\theta(\p_h)-\theta(\p'_1)-\cdots -\theta(\p'_h)|> \frac
1{7 |\p_1\cdots \p_h \p'_1\cdots \p'_h|}\ge
h/x,\end{equation*} which is a contradiction.\qedhere
\end{proof}

\section{Proof of Theorem \ref{thm:our_thm}}

We start following the lines of \cite{R} with several adjustments. In the sequel we will write
$\p$ for a Gaussian prime in the first octant ($0 < \theta(\p) < 1/8)$.

We fix  a number $c_h>h$ which will  determine the growth of the
sequence we construct. Indeed we will take $c_h=\sqrt{(h-1)^2+1}+(h-1)$.

\subsection{The construction}\label{s1} We will construct for each $\alpha\in [1,2]$ a sequence of positive integers indexed with the Gaussian primes
$$\B_{\alpha}:=\{b_{\p}\},$$ where each $b_{\p}$ will be built using the
development to base $2$ of  $\alpha \, \theta(\p)$:
\[
    \alpha \, \theta(\p) =  \sum_{i=1}^\infty \delta_{i \p} 2^{-i} \qquad (\delta_{i \p} \in \lbrace 0, 1 \rbrace).
\]

The role of the parameter $\alpha$ will be clear at a later stage,
for the moment it  is enough to note  that the set
$\{\alpha\theta(\p)\}$ obviously keeps the same $B_h$ property of
the set $\{\theta(\p)\}$.

To organize the construction  we  describe the sequence $\B_{\alpha}$ as a union of finite sets according with the sizes of the indexes:
$$\B_{\alpha}=\bigcup_{K}\B_{\alpha,K},$$
 where $$\B_{\alpha,K}=\{b_{\p}:\  \p\in P_K\}$$ and
$$P_K:=\{\p:\
 2^{\frac {(K-2)^2}{c_h}}\le |\p|^{2}<2^{\frac{(K-1)^2}{c_h}  }\}.$$

Now we build the positive integers $b_{\p}\in \B_{\alpha,K}$. For any $\p\in P_K$ we define $\widehat{\alpha \theta(\p)}$ the truncated series of $\alpha \, \theta(\p)$ at the
 $K^2$-place:
\begin{equation}\label{hat}
    \widehat{\alpha \theta(\p)} := \sum_{i=1}^{K^2} \delta_{i \p}
    2^{-i}
\end{equation}

and  combine the digits at places $(j-1)^2+1, \cdots, j^2$ into a
single number
\[
    \Delta_{j\p} = \sum_{i=(j-1)^2+1}^{j^2} \delta_{i \p} 2^{j^2-i} \quad (j=1, \cdots, K),
\]
so that we can write
\begin{equation} \label{eq:3.2}
   \widehat{\alpha \, \theta(\p)} = \sum_{j=1}^K \Delta_{j\p} 2^{-j^2}.
\end{equation}
We observe that if $\p\in P_K$ then
\begin{equation} \label{eq:3.1}
    \vert \widehat{\alpha \, \theta(\p)}- \alpha \, \theta(\p) \vert \le 2^{-K^2}.
\end{equation}
 The definition of $b_{\p}$ is
informally outlined as follows. We consider the series of blocks
$\Delta_{1\p},\cdots, \Delta_{K\p}$  and re-arrange them opposite to
the original left to right arrangement. Then we insert at the left
of each $\Delta_{j\p}$ an additional filling block of $2d+1$ digits,
with $d=\lceil \log_2 h \rceil$. At the filling blocks the digits
will be always $0$ except for only one exception:  in the middle of the
first filling block (placed to the left of the $\Delta_{K}$ block)
we put the digit $\textbf{1}$. This digit will mark which subset $P_K$ the
prime $\p$ belongs to.
 \[
  \alpha \, \theta(\p) = 0.\overset{\Delta_1}{1} \; \overset{\Delta_2}{\overbrace{001}} \dots
  \overset{\Delta_j}{\overbrace{1 \cdots \cdots 0}}  \dots \overset{\Delta_{K}}{\overbrace{01 \cdots \cdots \cdots  11 }}
  \underset{\overset{\uparrow}{K
  ^2}}{\ \dots} \dots
 \]
\[
  b_\p =  \underline{\mathbf{0 \cdot\mathbf{1}\cdot 0}} \; \overset{\Delta_{K}}{\overbrace{01 \cdots \cdots \cdots 11 }} \;
    \underline{\mathbf{0 \cdots 0}}
  \;\dots \underline{\mathbf{0 \cdots 0}}
  \;\overset{\Delta_j}{\overbrace{1 \cdots \cdots 0}} \;{\underline{\mathbf{0 \cdots 0}}\dots  \underline{\mathbf{0 \cdots
  0}}}
  \; \overset{\Delta_2}{\overbrace{001}} \; \underline{\mathbf{0 \cdots 0}}
  \; \overset{\Delta_1}{1},
 \]

The reason to add
the blocks of zeroes and the value of $d$ will be clarified just
before Lemma \ref{lem:3.1}.

\par More formally, for $\p\in P_K$ we define
\begin{equation} \label{eq:3.3}
    t_{\p}=2^{K^2+(2d+1)K+(d+1)},
\end{equation}
and
\[
    b_\p = t_\p + \sum_{j=1}^{K} \Delta_{j\p} 2^{(j-1)^2 + (2d+1)(j-1)}.
\]
Furthermore we define $\Delta_{j \p} = 0$ for $j > K $.

\begin{rem}The construction at \cite{R} was based on the numbers $\alpha \log p$, with $p$ rational prime,
 hence the digits of their integral parts had to be included also in the corresponding integers $b_p$.
 Ruzsa solved that problem reserving fixed places for these digits.
 Since in our construction the integral part of $\alpha \theta(\p)$ is zero we don't need to care about this.
 \end{rem}

\medskip
We observe that distinct primes $\p,\q$ provide distinct $b_\p,
b_\q$. Indeed if $b_\p=b_\q$ then
$\Delta_{i\p}=\Delta_{i\q}$ for all $i \le K$. Also $t_\p=t_\q$
which means $\p, \q \in P_K$, and so
\[
    \vert \theta(\p) -  \theta(\q) \vert = \alpha^{-1} \cdot \sum_{j > K} (\Delta_{j \p}- \Delta_{j \q}) < 2^{-K^2}.
\]
Now if $\p \ne \q$ then Lemma \ref{lem:desi} implies that $
    \vert \theta(\p)-\theta(\q) \vert >  \frac{1}{7 |\p \q|}  > 2^{-\frac 1{c}(K-1)^2 -
   3}.$
Combining both inequalities  we have a contradiction for  $K\ge h+1\ge 3$. So we assume $K\ge h+1$ through all the paper.

Since all the integers $b_{\p}$ are distinct, we have that
\begin{equation}\label{B}
|\B_{\alpha,K}|=|P_K|=\pi\left (2^{\frac{(K-1)^2}{c_h}};1,4\right )-\pi\left (2^{\frac{(K-2)^2}{c_h}};1,4\right )\\ \gg \frac{2^{\frac{K^2}{c_h}}}{K^2}.
\end{equation}
We observe also that $$b_{\p}<2^{K^2+(2d+1)K+(d+1)+1}.$$ Using these estimates  we can easily prove that  $\B_{\alpha}(x)=x^{\frac 1{c_h}+o(1)}$. Indeed, if $K$ is the integer such that
$2^{K^2+(2d+1)K+(d+1)+1}< x\le 2^{(K+1)^2+(2d+1)(K+1)+(d+1)+1}$ then we have
\begin{equation}\label{lowerbound}\B_{\alpha}(x) \ge |\B_{\alpha,K}|=
2^{\frac
1{c_h}K^2(1+o(1))}= x^{\frac 1{c_h}+o(1)}.\end{equation}
For the upper bound we have
$$\B_{\alpha}(x)\le \#\{\p:\ |\p|^2\le 2^{\frac{K^2}{c_h}}\}\le 2^{\frac{K^2}{c_h}}= x^{\frac 1{c_h}+o(1)}.$$

\par There is a compromise at the choice of a particular value of $c_h$ for the construction. On one hand larger values of $c_h$ capture more
information from the Gaussian arguments which brings the sequence $\B_{\alpha}=\{ b_\p \} $
closer to being a $B_h$ sequence. On the other hand smaller values
of $c_h$ provide higher growth of the counting function of $\B_{\alpha}$.

Clearly  $\B_{\alpha}$ would be a $B_h$ sequence if for all $l=2,
\cdots, h$ it does not contain  $b_{\p_1},\cdots,b_{\p_l},b_{\p_1'},\cdots,b_{\p_l'}$
satisfying
\begin{eqnarray}
    b_{\p_1} + \cdots + b_{\p_l} &=& b_{\p_1'} + \cdots + b'_l,\label{uplamala}\\ \lbrace b_1, \cdots, b_l \rbrace &\cap &\lbrace b'_1, \cdots, b'_l \rbrace = \emptyset,\nonumber \\
      b_{\p_1} \ge \cdots \ge b_{\p_l} \quad &\text{and}& \quad b_{\p'_1} \ge \cdots \ge b_{\p'_l}\label{eq:3.6}.
\end{eqnarray}
We say that $(\p_1,\dots, \p_l,\p_1',\dots,\p_l')$ is a bad $2l$-tuple if the equation \ref{uplamala} is satisfied by the corresponding $b_{\p_r}$.

The sequence $\B_{\alpha}=\{b_{\p}\}$ we have constructed is not properly a
$B_h$ sequence. Some repeated sums as in \eqref{uplamala} will eventually appear, however
the precise way how the elements $b_{\p}$ are built will allow us to study these bad $2l$-tuples and to prove that there are not too many repeated sums. 
Then after removing the bad elements involved in these bad $2l$-tuples we will obtain a true $B_h$ sequence.

Now we will see why blocks of zeroes were added to the binary
development of $b_{\p}$. We can identify each $b_{\p}$ with a vector
 as follows:
\begin{align*}
    b_{\p_1} & = \left(\cdots,\textbf{1},\Delta_{K_1 \p_1},0, \cdots,0,\Delta_{K_2 \p_1},0, \cdots, 0,\Delta_{K_l \p_1},0, \cdots ,0,\Delta_{2 \p_1}, 0, \Delta_{1 \p_1} \right) \\
    b_{\p_2} & = \left(\cdots,0,\cdots \cdots \cdots \cdots \cdot,   \textbf{1},\Delta_{K_2 \p_2},0, \cdots, 0,\Delta_{K_l \p_2},0, \cdots,0, \Delta_{2 \p_2}, 0, \Delta_{1 \p_2}\right)\\
             &   \qquad \vdots \qquad \qquad \qquad \qquad \qquad \qquad  \vdots \qquad \qquad \qquad \qquad \qquad \qquad  \vdots\\
    b_{\p_l} & = \left(\cdots,0, \cdots\cdots\cdots\cdots \cdot ,0,\cdots\cdots\cdots\cdots\cdot,    \textbf{1},\Delta_{K_l \p_l},0, \cdots,0, \Delta_{2 \p_l}, 0, \Delta_{1 \p_l}\right),
    %b_{\p_3} &= \left(\Delta_{1 \p_3},0, \Delta_{2 \p_3},0, \cdots, \Delta_{K_3 \p_3},\textbf{1}, \cdots, 0 \quad \quad \; ,0, \cdots, 0, \quad \quad   ,0, 0, \cdots    \right)\\ \cdots
 \end{align*}
where each comma represents one block of $d$ zeroes. Note that the leftmost part of each vector is null.
The value of $d=\lceil \log_2 h\rceil$ has been chosen to prevent the propagation of
the carry between any two consecutive coordinates separated by a comma in the above identification. So when we sum no more than $h$
integers $b_{\p}$ we can just sum the corresponding vectors coordinate-wise.
This argument implies the following lemma.

\begin{lem} \label{lem:3.1}
    Let $(\p_1, \cdots, \p_l , \p'_1, \cdots, \p'_l )$ be a bad $2l$-tuple. %satisfying \eqref{eq:3.4},  \eqref{eq:3.5} and \eqref{eq:3.6}.
    Then there are integers $K_1 \ge \cdots \ge K_l$ such that $\p_1,\p'_1 \in P_{K_1}, \cdots,  \p_l, \p'_l \in P_{K_l},$ and we have
    \begin{equation} \label{eq:3.7}
        \widehat{\alpha \theta(\p_1)} + \cdots + \widehat{\alpha \theta(\p_l)} = \widehat{\alpha \theta(\p'_1)} + \cdots + \widehat{\alpha \theta(\p'_l)} .
    \end{equation}
\end{lem}
\begin{proof}
    Note that \eqref{uplamala} implies
    $ t_{\p_1} + \cdots + t_{\p_l} = t_{\p'_1} + \cdots + t_{\p'_l}$ and  $\Delta_{j \p_1} + \cdots + \Delta_{j \p_l} =
    \Delta_{j \p'_1} + \cdots + \Delta_{j \p'_l}$ for each $j$.
   Using \eqref{eq:3.2} we conclude \eqref{eq:3.7}.
     As the bad $2l$-tuple satisfies condition \eqref{eq:3.6} we deduce that $\p_r,\p_r'$ belongs to the same $P_{K_r}$ for all $r$.
  \end{proof}

According to the previous lemma we will write $E_{2l}(\alpha;K_1,\cdots ,K_l)$ for the set of bad $2l$-tuples $(\p_1,\cdots,\p_l')$ with $\p_r,\p'_r\in P_{K_r},\ 1\le r\le l$
and
$$ E_{2l}(\alpha;K)=\bigcup_{K_l\le \cdots \le
K_1=K} E_{2l}(\alpha;K_1,\cdots,K_l),$$
where $K=K_1$. Also we define the set
$$\text{Bad}_{\alpha,K}=\{b_{\p}\in \B_{\alpha,K}:\ b_{\p} \text{ is the largest element involved in some equation \ref{uplamala}}\}.$$
It is clear that $\sum_{l\le h}|E_{2l}(\alpha,K)|$ is an upper bound for $|\text{Bad}_{\alpha,K}|$, the number of elements we need to remove from each $\B_{\alpha,K}$ to get a $B_h$ sequence.

We do know how to obtain a good upper bound for $|E_{2l}(\alpha,K)|$ for a particular $\alpha$, but we can do it for almost all $\alpha$.

\begin{lem}\label{inte} For $l=2,3,4$ we have
$$\int_{1}^2 |E_{2l}(\alpha,K)|\ud\alpha\ll K^{m_l}2^{\left (\frac{2(l-1)}{c_h}-1\right )(K-1)^2-2K}$$for some $m_l$.
\end{lem}

The proof this lemma is involved and we postpone it to  section \S 4. 
We think that Lemma \ref{inte} holds for any $l$ but we have not found a proof yet. %Indeed, each case of Lemma \ref{inte} will be proved separately.

\subsection{Last step in the proof of the theorem}

For $h=2,3,4$ we have that
\begin{eqnarray*}
\int_1^2\sum_K\frac{|\text{Bad}_{\alpha,K}|}{|\B_{\alpha,K}|}\ud\alpha&\ll & \sum_K\frac{\sum_{l\le h}\int_1^2|E_{2l}(\alpha,K)|\ud\alpha}{K^{-2}2^{\frac 1{c_h}(K-1)^2}}\\
&\ll &\sum_K \frac{\sum_{l\le h}K^{m_l}2^{\left (\frac{2(l-1)}{c_h}-1\right )(K-1)^2-2K}}{K^{-2}2^{\frac 1{c_h}(K-1)^2}}\\
&\ll &\sum_K K^{m_l+2}2^{\left (\frac{2(h-1)}{c_h}-1-\frac 1{c_h}\right )(K-1)^2-2K}.
\end{eqnarray*}

The last sum is finite for $c_h=\sqrt{(h-1)^2+1}+(h-1)$  which is the largest number for which $\frac{2(h-1)}{c_h}-1-\frac 1{c_h}\le 0$. So for this $c_h$ the sum $\sum_K\frac{|\text{Bad}_{\alpha,K}|}{|\B_{\alpha,K}|}$ is convergent for almost all $\alpha\in[1,2]$. We take one of these $\alpha$, say $\alpha_0$, and consider the sequence
$$\B=\bigcup_K\left (\B_{\alpha_0,K}\setminus \text{Bad}_{\alpha_0,K}\right ).$$
We claim that  this sequence satisfies the condition of the theorem. 
On one hand this sequence clearly is a $B_h$ sequence because we have destroyed all the repeated sums of $h$ elements of $\B_{\alpha_0}$  removing all the bad elements from each $\B_{\alpha_0,K}$.

On the other hand the convergence of $\sum_K\frac{|\text{Bad}_{\alpha_0,K}|}{|\B_{\alpha,K}|}$ implies that $|\text{Bad}_{\alpha_0,K}|=o\left (|\B_{\alpha,K}| \right )$. We proceed as in \eqref{lowerbound} to estimate the counting function of $\B$. For any $x$ let $K$ the integer such that
$2^{K^2+(2d+1)K+(d+1)+1}< x\le 2^{(K+1)^2+(2d+1)(K+1)+(d+1)+1}.$ We have
$$\B(x) \ge |\B_{\alpha_0,K}|-|\text{Bad}_{\alpha_0,K}|=
|\B_{\alpha_0,K}|(1+o(1))\gg K^{-2}2^{\frac
1{c_h}K^2}= x^{\frac 1{c_h}+o(1)}.$$
For the  upper bound, we have
$$\B(x)\le \B_{\alpha_0}(x)=x^{\frac 1{c_h}+o(1)}.$$
Hence  $$\B(x)=x^{\sqrt{(h-1)^2+1}-(h-1)+o(1)}.$$

\section{Proof of Lemma \ref{inte}}
The proof of Lemma \ref{inte} will be a consequence of Propositions  \ref{prop:E_4-in_average}, \ref{prop:E_6-in_average} and \ref{prop:E_8-in_average}.
Before proving these propositions we need some properties of the bad $2l$-tuples  and an auxiliary lemma about visible lattice points.
\
\subsection{Some properties of the $2l$-tuples}
For any $2l$-tuple  $(\p_1, \cdots, \p'_l )$
%,\K_1\ge \cdots K_l
we define the numbers $\omega_s=\omega_s(\p_1, \cdots, \p'_l)$ by
\begin{equation*}
\omega_s=\sum_{r=1}^s\left (\theta(\p_r)-\theta(\p_r')\right) \qquad (s \le l).
\end{equation*}
The next two lemmas contain several  properties of the bad $2l$-tuples.

\begin{lem}\label{condiciones}
    Let $(\p_1, \cdots, \p_l , \p'_1, \cdots, \p'_l )$ be a
    bad
    $2l$-tuple
    with $K_1 \ge \cdots \ge K_l$ given by Lemma \ref{lem:3.1}. We have %\done\todo{poner estas condiciones en modo displaymath}
\begin{align*}
 \mathrm{i)} \quad	 	& \left \vert \omega_l \right \vert \,  \le \, l2^{-K_l^2 },\\
 %\mathrm{ii)}\quad 	&  |\omega_s|\ge 2^{-\frac 1{c_h}((K_1-1)^2+\dots +(K_s-1)^2)-3},\qquad s=1,\cdots, l, \\
 \mathrm{ii)} \quad	& |\omega_{l-1}|\ge 2^{-\frac{1}{c_h}(K_l-1)^2-4}, \\
 \mathrm{iii)} \quad 	& (K_l-1)^2 \le  \frac{(K_{1}-1)^2+ \cdots + (K_{l-1}-1)^2}{c_h - 1}.
\end{align*}

\end{lem}
\begin{proof}
   i) This is a consequence of  \eqref{eq:3.7} and \eqref{eq:3.1}:
    \[
       |\omega_l|=\frac 1{\alpha} \left \vert \sum_{r=1}^l (\alpha\theta(\p_r) - \alpha\theta(\p'_r)) \right \vert\le \frac 1{\alpha}
       \left (  2^{-K_1^2} + \cdots + 2^{-K_l^2}\right )
        \le
        l2^{-K_l^2}.
    \] %since $\alpha\ge 1$.

 %  ii) This comes from Lemma \ref{lem:desi} which implies
%    \begin{align*}
%       |\omega_s|=\left \vert \sum_{r=1}^s\left (\theta ( \p_r)-\theta( \p'_r) \right ) \right \vert \
%        &> \frac{1}{ 7 \left \vert  \p_1 \cdots  \p'_s \right \vert } >  2^{-3-\frac 1{c_h}\sum_{r=1}^l(K_r-1)^2},
%    \end{align*}
ii) Lemma \ref{lem:desi}  implies
\begin{equation}
|\theta(\p_l)-\theta(\p_l')|\ge \frac{1}{ 7 \left \vert  \p_l \p'_l \right \vert }\ge 2^{-3-\frac 1{c_h}(K_l-1)^2}
\end{equation}
   % \begin{align*}
%       |\omega_s|=\left \vert \sum_{r=1}^s\left (\theta ( \p_r)-\theta( \p'_r) \right ) \right \vert \
%        &> \frac{1}{ 7 \left \vert  \p_1 \cdots  \p'_s \right \vert } >  2^{-3-\frac 1{c_h}\sum_{r=1}^l(K_r-1)^2},
%    \end{align*}
    and so,
\begin{align*}
|\omega_{l-1}|	&=|\omega_l+\theta(\p_l')-\theta(\p_l)| \ge |\theta(\p_l')-\theta(\p_l)|-|\omega_l|\\
    						&\ge 2^{-\frac 1{c_h}(K_l-1)^2-3}-l2^{-K_l^2}\ge 2^{-\frac 1{c_h}(K_l-1)^2-4},
\end{align*}
since $ K_l\ge h+1\ge l+1$.

   iii)  Lema \ref{lem:desi} also implies that
    \begin{align*}
       |\omega_l|=\left \vert \sum_{r=1}^l\left (\theta ( \p_r)-\theta( \p'_r) \right ) \right \vert \
        &> \frac{1}{ 7 \left \vert  \p_1 \cdots  \p'_l \right \vert } >  2^{-3-\frac 1{c_h}\sum_{r=1}^l(K_r-1)^2}.
    \end{align*}
   Combining this with i) we obtain
   \begin{eqnarray*}(K_l-1)^2\le \frac 1{c_h-1}\left ((K_1-1)^2+\cdots +(K_{l-1}-1)^2\right )+\frac{\log_2l-2K_l+4}{1-1/c_h}.  \end{eqnarray*}
   The last term is negative because $K_l\ge h+1\ge l+1$ and $l\ge 2$.
\end{proof}

\begin{lem} \label{lem:badtuple-second-condition}
    Suppose that $( \p_1, \cdots, \p_l , \p'_1, \cdots, \p'_l)$ is a bad $2l$-tuple. \\ % which satisfies \eqref{eq:3.4} and \eqref{eq:3.5}.
     Then for any  $\omega_s=\sum_{r=1}^{s}\left (\theta(\p_r)- \theta(\p'_r)\right)$ with $1 \le s \le l-1$ we have
    \begin{equation} \label{eq:badtuple-second-condition}
        \left \|\alpha2^{K_{s+1}^2}\omega_s
        \right \|  \le s2^{K_{s+1}^2-K_s^2} \quad \qquad (s=1,\cdots, l-1),
    \end{equation}
    where $\Vert \cdot \Vert$ means the distance to the nearest integer.
      \end{lem}
\begin{proof}Since $0\le
\alpha\theta(\p)-\widehat{\alpha\theta(\p)}\le 2^{-K^2}$ when $\p\in
P_K$, we can write
\begin{eqnarray*}2^{K_{s+1}^2} \alpha \, \sum_{r=1}^{s}\left (\theta(\p_r)-
\theta(\p'_r)\right )= 2^{K_{s+1}^2} \sum_{r=1}^{s}\left
(\widehat{\alpha\theta(\p_r)} - \widehat{\alpha\theta(\p'_r)}\right
)+\epsilon_s,
\end{eqnarray*}
with $ |\epsilon_s|\le s2^{K_{s+1}^2-K_s^2}$. By Lemma
\ref{lem:3.1} we know that $\sum_{r=1}^{l}\left
(\widehat{\alpha\theta(\p_r)}- \widehat{\alpha\theta(\p'_r)}\right
)=0$ when $( \p_1, \cdots, \p_l , \p'_1, \cdots, \p'_l)$ is a
bad $2l$-tuple. %which satisfies \eqref{eq:3.4} and \eqref{eq:3.5}.
Using this and \eqref{hat} we have that
\begin{eqnarray*} 2^{K_{s+1}^2} \sum_{r=s+1}^{l}\left
(\widehat{\alpha\theta(\p'_r)}- \widehat{\alpha\theta(\p_r)}\right )
=
\sum_{r=s+1}^{l}\sum_{i=1}^{K^2_r}2^{K_{s+1}^2-i}(\delta_{i\p'_r}-\delta_{i\p_r})
\end{eqnarray*}
is an integer, which proves \eqref{eq:badtuple-second-condition}.

\end{proof}

\begin{lem}\label{medida}
    $$\int_1^2 \vert E_{2l}(\alpha;K_1,\cdots,K_l) \vert \ud\alpha\ll  2^{K_l^2-K_1^2}\sum_{\substack{(\p_1,\cdots,\p_l')\\|\omega_l|<l\cdot 2^{-K_l^2}}}\frac{|\omega_{l-1}|}{|\omega_1|}\prod_{j=1}^{l-2}\left (\frac{|\omega_{j}|}{|\omega_{j+1}|}+1  \right )$$
\end{lem}
\begin{proof}  We have seen that if $(\p_1,\cdots,\p_{l}')\in E_{2l}(\alpha;K_1,\cdots,K_l)$, then
\begin{equation} %\label{eq:badtuple-second-condition}
        \left \|\alpha2^{K_{s+1}^2}\omega_s
        \right \|  \le s2^{K_{s+1}^2-K_s^2},\ s=1,\dots, l-1.
    \end{equation} Then there exists integers $j_s,\ s=1,\cdots,l-1$ such that
    \begin{eqnarray}\label{interval}\left |\alpha-\frac{j_s}{2^{K_{s+1}^2}\omega_s}\right |\le \frac{s2^{-K_s^2}}{|\omega_s|}.\end{eqnarray}
   Writing $I_{j_1},\cdots, I_{j_s}$ for the intervals defined by the inequalities \ref{interval}, we have
    \begin{eqnarray*}\mu\{\alpha:\ (\p_1,\cdots,\p_l')\in E_{2l}(\alpha;K_1,\cdots, K_l)\}&\le& \sum_{j_1,\cdots, j_{l-1}}\left |I_{j_1}\cap \cdots \cap I_{j_{l-1}}\right |\\
   & \le& \frac{2^{-K_1^2+1}}{|\omega_1|}\#\{(j_1,\dots, j_{l-1}):\ \bigcap_{i=1}^{l-1}I_{j_i}\ne \emptyset\}\end{eqnarray*}
   To estimate this last cardinal note that for all $s=1,\cdots, l-2$ we have
   $$\left |\frac{j_s}{2^{K_{s+1}^2}\omega_s}-\frac{j_{s+1}}{2^{K_{s+2}^2}\omega_{s+1}}\right |<\left |\alpha-\frac{j_s}{2^{K_{s+1}^2}\omega_s}\right |+
   \left |\alpha-\frac{j_{s+1}}{2^{K_{s+2}^2}\omega_{s+1}}\right |<
   \frac{s2^{-K_s^2}}{|\omega_s|}+\frac{(s+1)2^{-K_{s+1}^2}}{|\omega_{s+1}|}$$
   Thus,
   \begin{eqnarray}\label{de}\left |j_s-j_{s+1}\frac{2^{K_{s+1}^2}\omega_s}{2^{K_{s+2}^2}\omega_{s+1}}\right |<
   s2^{-K_s^2+K_{s+1}^2}+\frac{(s+1)|\omega_{s}|}{|\omega_{s+1}|}.\end{eqnarray}
   We observe that for each $s=1,\cdots,l-2$ and for each $j_{s+1}$, the number of $j_s$ satisfying \eqref{de} 
   is bounded by $2\left ( s2^{-K_s^2+K_{s+1}^2}+\frac{(s+1)|\omega_s|}{|\omega_{s+1}|}   \right )+1\ll \frac{|\omega_{s}|}{|\omega_{s+1}|}+1$.

   Note also that \begin{eqnarray*}|j_{l-1}|&\le& 2^{K_l^2}|\omega_{l-1}|\left (\left |\frac{j_{l-1}}{2^{K_l^2}\omega_{l-1}}-\alpha\right |+|\alpha| \right )\\ &\le & 2^{K_l^2}|\omega_{l-1}|\left (\frac{(l-1)2^{K_{l-1}^2}}{|\omega_{l-1}|}+2\right )\\
   &\le &l-1+ 2^{K_l^2+1}|\omega_{l-1}|\\ &\ll & 2^{K_l^2+1}|\omega_{l-1}|.\end{eqnarray*}
   %In the last step we have used the condition iii).

   The proof can be completed putting all these observations together.
   \end{proof}

   \subsection{Visible points} We will denote by $\mathcal V$ the set of lattice points visible from the origin except  $(1,0)$. In the next subsection we will use several times the following lemma.
\begin{lem} \label{lem:visible}
    The number of integral lattice points visible from the origin that are contained in a circular sector centred at the origin of radius $R$
    and angle $\epsilon$ is at most $\epsilon R^2 +1$. In other words, for any real number $t$
    \[
    	 \# \lbrace \nu \in \mathcal V, \; \vert \nu \vert < R, \; \Vert \theta(\nu) + t \Vert < \epsilon \rbrace \le \epsilon R^2+1.
    \]
    Furthermore,
     \[
    	 \# \lbrace \nu \in \mathcal V, \; \vert \nu \vert < R, \; \Vert \theta(\nu) \Vert < \epsilon \rbrace \le \epsilon R^2.
    \]
 \end{lem}
\begin{proof}
    We arrange the $N$ lattice points inside de sector $\nu_1, \nu_2, \cdots, \nu_N$ that are visible from the origin $O$ by the value of their argument so that $\theta(\nu_i) < \theta(\nu_j) $ for $1 \le i < j \le N$.
    For each $i=1,\dots, N-1$ the three lattice points $O, \nu_i, \nu_{i+1}$ define a triangle $T_i$ with  $\text{Area}(T_i)\ge 1/2$, that does not contain any other lattice
    point.
         \par Since all $T_i$ are inside the circular sector their union covers at most the area of the sector. They don't overlap pairwise, thus
    \[
        N-1 \le \sum_{i=1}^N2\cdot \text{Area}(T_i) =  2\cdot \text{Area}\left (\bigcup_{i=1}^NT_i\right )\le R^2 \epsilon.
    \]
    For the last statement we add $\nu_0=(1,0)$ to our $N$ visible points $\nu_1,\dots, \nu_N$ and we repeat the argument.
\end{proof}
\subsection{Estimates for the number of bad $2l$-tuples ($l=2,3,4$)}\label{s3}

We start with the case $l=2$ which was considered by Ruzsa
for $B_2$ sequences. In the sequel all lattice points $\nu$ appearing in the proofs belong to $\mathcal V$ and Lemma \ref{lem:visible} applies.
\begin{prop} \label{prop:E_4-in_average} For any $c_h>2$ we have
$$\int_1^2 |E_4(\alpha;K)| \ud{\alpha}\ll K 2^{(\frac{2}{c_h-1}-1)(K-1)^2-2K}.$$
\end{prop}
\begin{proof}
 Lemma \ref{medida} implies that
\[
\int_1^2 |E_{4}(\alpha;K_1,K_2)| \ud{\alpha} \ll 2^{K_2^2-K_1^2} \, \#\{(\p_1,\p_1',\p_2,\p_2') \colon |\omega_2| \le 2\cdot 2^{-K_2^2}\}.
\]
We get an upper bound for the second factor here by using Lemma \ref{lem:visible}
to estimate the number of lattice points of the form $\nu_2= \p_1\p_1' \overline{\p_2\p_2'}$ such that $\Vert \theta(\nu_2) \Vert < \epsilon, \, \vert \nu_2 \vert < R $,
with $\epsilon = 2\cdot 2^{-K_2^2}$ and $R= 2^{\frac{1}{c_h}((K_1-1)^2+(K_2-1)^2)}$.
 We have
\begin{eqnarray*}
 \int_1^2 |E_{4}(\alpha;K_1,K_2)| \ud{\alpha} &\ll& 2^{K_2^2-K_1^2}\cdot 2^{\frac{2}{c_h}((K_1-1)^2+(K_2-1)^2)-K_2^2}\\
 &\ll & 2^{\frac{2}{c_h}((K_1-1)^2+(K_2-1)^2)-K_1^2  }.\end{eqnarray*}
By Lemma \ref{condiciones} iii) we also have $(K_2-1)^2\le
\frac{(K_1-1)^2}{c_h-1}$, thus
$$ \int_1^2 |E_{4}(\alpha;K_1,K_2)| \ud{\alpha}\ll 2^{\left (\frac 2{c_h-1}-1\right )K_1^2-2K_1}  $$ and
 \begin{eqnarray*}\int_1^2 | E_4(\alpha;K)| \ud{\alpha}=\sum_{K_2 \le K} \int_1^2 |E_4(\alpha;K,K_2)|\ud{\alpha}\ll  K2^{\left (\frac{2}{c_h-1}-1\right )(K-1)^2-2K}.\end{eqnarray*}
\end{proof}

\begin{prop} \label{prop:E_6-in_average} For any $c_h >3$ we have
$$\int_1^2 |E_6(\alpha;K)| \ud{\alpha} \ll K^2 \, 2^{(\frac{4}{c_h-1}-1)(K-1)^2-2K}.$$
\end{prop}
\begin{proof}
Lemma \ref{medida} says that
\begin{eqnarray*}\label{lem:3.5}\int_1^2 |E_6(\alpha;K_1,K_2,K_3)| \ud{\alpha}\ll 2^{K_3^2-K_1^2} \sum_{\substack{(\p_1,\cdots,\p_3')\\|\omega_3|\le 3\cdot 2^{-K_3^2}}} \frac{1}{|\omega_1|}.
\end{eqnarray*}
Applying Lemma \ref{lem:visible} by writing $\nu_1=\p_1\overline{\p_1'}$ and $\nu_2=\p_2\p_3\overline{\p_2'}\overline{\p_3'}$, we have that
\begin{eqnarray*}\sum_{\substack{(\p_1,\cdots,\p_3')\\|\omega_3|\le 3\cdot 2^{-K_3^2}}} \frac{1}{|\omega_1|} &\ll &\sum_m2^m\#\{(\p_1,\cdots,\p_3'):\ |\omega_1|\le 2^{-m},|\omega_3|\le 3\cdot 2^{-K_3^2}\}\\ &\ll&\sum_m2^m\#\{(\nu_1,\nu_2):\ \| \theta(\nu_1)\|\le 2^{-m},\ \|\theta(\nu_1)+\theta(\nu_2)\|\le 3\cdot 2^{-K_3^2}\}\\
&\ll &\sum_m2^m \sum_{|\theta(\nu_1)|\le 2^{-m}} \# \{\nu_2:\  \|\theta(\nu_1)+\theta(\nu_2)\|\le 3\cdot 2^{-K_3^2}\}\\
&\ll &\sum_m 2^m\cdot 2^{\frac 2{c_h}(K_1-1)^2-m}\left (2^{\frac 2{c_h}\left ((K_2-1)^2+(K_3-1)^2\right )-K_3^2}+1\right ).
\end{eqnarray*}
Hence using the inequalities $K_3\le K_2\le K_1$ and  $(K_3-1)^2\le \frac{(K_2-1)^2+(K_1-1)^2}{c_h-1}$ (Lemma \ref{condiciones} iii))  we have
\begin{eqnarray*}
 \int_1^2 |E_6(\alpha;K_1,K_2,K_3)| \ud{\alpha} &\ll& K_1^22^{K_3^2-K_1^2+\frac 2{c_h}(K_1-1)^2}\left (2^{\frac 2{c_h}\left ((K_2-1)^2+(K_3-1)^2\right )-K_3^2}+1\right )\\
&\ll& K_1^22^{-K_1^2+\frac 2{c_h}\left ((K_1-1)^2+(K_2-1)^2+(K_3-1)^2\right )}+K_1^22^{K_3^2-K_1^2+\frac 2{c_h}(K_1-1)^2}\\
&\ll& K_1^22^{-(K_1-1)^2+\frac 2{c_h}\left ((K_1-1)^2+(K_2-1)^2+(K_3-1)^2\right )-2K_1}\\ & &+\quad K_1^22^{(K_3-1)^2-(K_1-1)^2+\frac 2{c_h}(K_1-1)^2}\\
&\ll& K_1^22^{\left (\frac 4{c_h-1}-1\right )(K_1-1)^2-2K_1}+K_1^22^{\left (\frac 4{c_h-1}-1\right )(K_1-1)^2-\frac 2{c_h(c_h-1)}(K_1-1)^2}\\
&\ll& K_1^22^{\left (\frac 4{c_h-1}-1\right )(K_1-1)^2-2K_1}.
\end{eqnarray*}

%Summing in all $m$, we have
%\begin{eqnarray*}
%\int_1^2 |E_6(\alpha;K_1,K_2,K_3)| \ud{\alpha} \ll K_1^2 2^{\left (\frac 4{c_h-1}-1\right )(K_1-1)^2-2K_1}
%\end{eqnarray*}

Then we can write
\begin{align*}
\int_1^2 |E_6(\alpha;K)| \ud{\alpha} &=\sum_{K_3\le K_2\le K}\int_1^2 |E_6(\alpha;K,K_2,K_3)| \ud{\alpha}\ll  K^4 2^{\left (\frac
4{c-1}-1\right )(K-1)^2-2K},
\end{align*}
as claimed.
\end{proof}

\begin{prop} \label{prop:E_8-in_average} For any $c_h >4$ we have
$$\int_1^2 | E_8(\alpha;K)| \ud{\alpha} \ll K^2 \, 2^{(\frac{6}{c_h-1}-1)(K-1)^2-2K}.$$
\end{prop}

\begin{proof} Considering the two possibilities $|\omega_1|<|\omega_2|$ and $|\omega_1|\ge |\omega_2|$ we get the inequality $\frac{|\omega_3|}{|\omega_1|}\left (\frac{|\omega_1|}{|\omega_2|}+1\right )\left (\frac{|\omega_2|}{|\omega_3|}+1\right )\ll \frac{|\omega_3|}{|\omega_1|}\left (\frac{|\omega_1|}{|\omega_2|}+1\right )\frac{1}{|\omega_3|}\ll \max \left (\frac 1{|\omega_1|},\frac 1{|\omega_2|}\right )$.
This combined with Lemma \ref{medida} implies that
\begin{eqnarray*}\int_1^2|E_8(\alpha,K_1,K_2,K_3,K_4)|\ud\alpha &\ll &2^{-K_1^2+K_4^2}\left (\sum_{\substack{(\p_1,\cdots, \p_4')\\|\omega_4|\le 4\cdot 2^{-K_4^2}}}\frac 1{|\omega_1|}+\sum_{\substack{(\p_1,\cdots, \p_4')\\|\omega_4|\le 4\cdot 2^{-K_4^2}}}\frac 1{|\omega_2|}\right )\end{eqnarray*}

%We observe that
%\begin{equation*}
%\sum_{\substack{(\p_1,\dots, \p_4')\\|\omega_4|\le 4\cdot 2^{-K_4^2}}}\frac 1{|\omega_1|}\ll \sum_m2^m\#\{(\p_1,\dots ,\overline{\p_4}):\ |\omega_1|<2^{-m},\ |\omega_4|\le 4\cdot 2^{-K_4^2}\}.
%\end{equation*}

Applying Lemma \ref{lem:visible} with the notation $\nu_1=\p_1\overline{\p_1'}$ and $\nu_2=\p_2\p_3\p_4\overline{\p_2'}\overline{\p_3'}\overline{\p_4'}$, we have that
\begin{eqnarray*}\sum_{\substack{(\p_1,\cdots, \p_4')\\|\omega_4|\le 4\cdot 2^{-K_4^2}}}\frac 1{|\omega_1|}&\ll &\sum_m2^m\#\{(\p_1,\cdots ,\overline{\p_4}):\ |\omega_1|<2^{-m},\ |\omega_4|\le 4\cdot 2^{-K_4^2}\}\\
&\ll & \sum_m2^m\#\{(\nu_1,\nu_2):\ \|\theta(\nu_1)\|\le 2^{-m},\ \|\theta(\nu_1)+\theta(\nu_2)\|\le 4\cdot 2^{-K_4^2}\}\\
&\ll &  \sum_m \sum_{\|\theta(\nu_1)\|<2^{-m}} \#\{\nu_2:\  \|\theta(\nu_1)+\theta(\nu_2)\|\le 4\cdot 2^{-K_4^2}\}                 \\
&\ll &\sum_m 2^{\frac 2{c_h}(K_1-1)^2}\left (2^{\frac 2{c_h}((K_2-1)^2+(K_3-1)^2+(K_4-1)^2)-K_4^2}+1\right )\\
& \ll & K_1^22^{\frac 2{c_h}((K_1-1)^2+(K_2-1)^2+(K_3-1)^2+(K_4-1)^2)-K_4^2} +K_1^22^{\frac 2{c_h}(K_1-1)^2}.\end{eqnarray*}

Similarly, but writing now $\nu_1=\p_1\p_2\overline{\p_1'}\overline{\p_2'}$ and $\nu_2=\p_3\p_4\overline{\p_3'}\overline{\p_4'}$ we have
\begin{eqnarray*}\sum_{\substack{(\p_1,\cdots, \p_4')\\|\omega_4|\le 4\cdot 2^{-K_4^2}}}\frac 1{|\omega_2|}&\ll &\sum_m2^m\#\{(\p_1,\cdots ,\overline{\p_4}):\ |\omega_2|\le 2^{-m},\ |\omega_4|\le 4\cdot 2^{-K_4^2}\}\\
&\ll & \sum_{m\le K_4^2}2^m\#\{(\nu_1,\nu_2):\ \|\theta(\nu_1)\|\le 2^{-m},\ \|\theta(\nu_1)+\theta(\nu_2)\|\le 4\cdot 2^{-K_4^2}\}\\
&+ & \sum_{m> K_4^2}2^m\#\{(\nu_1,\nu_2):\ \|\theta(\nu_1)\|\le 2^{-m},\ \|\theta(\nu_1)+\theta(\nu_2)\|\le 4\cdot 2^{-K_4^2}\}\\
&=&S_1+S_2 \end{eqnarray*}
We observe that if $m\le K_4^2$ then $\|\theta(\nu_2)\|\le \|\theta(\nu_1)+\theta(\nu_2)\|+\|\theta(\nu_1)\|\le 5\cdot 2^{-m}$. Thus
\begin{eqnarray*}S_1&\ll &\sum_{m\le K_4^2}2^m\#\{(\nu_1,\nu_2):\ \|\theta(\nu_2)\|\le 5\cdot 2^{-m},\ \|\theta(\nu_1)+\theta(\nu_2)\|\le 4\cdot 2^{-K_4^2}\}\\
&\ll &\sum_m 2^m\sum_{\|\theta(\nu_2)\|\le 5\cdot 2^{-m}}\#\{\nu_1:\ \|\theta(\nu_1)+\theta(\nu_2)\|\le 4\cdot 2^{-K_4^2}\}\\
& \ll & \sum_m 2^m\cdot 2^{\frac 2{c_h}\left ((K_3-1)^2+(K_4-1)^2\right )-m}\left (2^{\frac 2{c_h}\left ((K_1-1)^2+(K_2-1)^2\right )-K_4^2}+1 \right )\\
& \ll & K_1^22^{\frac 2{c_h}((K_1-1)^2+(K_2-1)^2+(K_3-1)^2+(K_4-1)^2)-K_4^2} +K_1^22^{\frac 2{c_h}\left ((K_3-1)^2+(K_4-1)^2\right )}.
\end{eqnarray*}

To estimate $S_2$, we observe that if $m>K_4^2$ then $\|\theta(\nu_2)\|\le \|\theta(\nu_1)+\theta(\nu_2)\|+\|\theta(\nu_1)\|\le 5\cdot 2^{-K_4^2}.$ Thus
\begin{eqnarray*}S_2&\ll &\sum_{m> K_4^2}2^m\#\{(\nu_1,\nu_2):\ \|\theta(\nu_1)\|\le 2^{-m},\ \|\theta(\nu_2)\|\le 5\cdot 2^{-K_4^2}\}\\
& \ll & \sum_m 2^m\cdot 2^{\frac 2{c_h}\left ((K_1-1)^2+(K_2-1)^2\right )-m}\cdot 2^{\frac 2{c_h}\left ((K_3-1)^2+(K_4-1)^2\right )-K_4^2}\\
& \ll & K_1^22^{\frac 2{c_h}((K_1-1)^2+(K_2-1)^2+(K_3-1)^2+(K_4-1)^2)-K_4^2}.
\end{eqnarray*}
Putting together the estimates we have obtained for $\sum \frac 1{|\omega_1|}$ and $\sum \frac 1{|\omega_2|}$ we get
\begin{eqnarray*}\int_1^2|E_8(\alpha,K_1,K_2,K_3,K_4)|\ud\alpha &\ll &  K_1^22^{\frac 2{c_h}((K_1-1)^2+(K_2-1)^2+(K_3-1)^2+(K_4-1)^2)-K_1^2}\\
 &+&K_1^22^{-K_1^2+K_4^2+\frac 2{c_h}(K_1-1)^2}\\&+& K_1^22^{K_4^2-K_1^2+\frac 2{c_h}\left ((K_3-1)^2+(K_4-1)^2\right )}\\ &=& T_1+T_2+T_3.\end{eqnarray*}

 Using the inequalities $(K_4-1)^2\le \frac 1{c_h-1}\left ((K_1-1)^2+(K_2-1)^2+(K_3-1)^2\right )$ and $K_4\le K_3\le K_2\le K_1$ we have
 \begin{eqnarray*}
 T_1&\ll &K_1^2 2^{\left (-1+\frac 6{c_h-1}\right )(K_1-1)^2-2K_1}\\
 T_2&\ll& K_1^22^{-(K_1-1)^2+(K_4-1)^2+\frac 2{c_h}(K_1-1)^2}\\
 &\ll& K_1^22^{\left (-1+\frac 3{c_h-1}+\frac 2{c_h}\right )(K_1-1)^2}\\
 &\ll &K_1^2 2^{\left (-1+\frac 6{c_h-1}\right )(K_1-1)^2-2K_1}.
\\
 T_3&\ll & K_1^22^{(K_4-1)^2-(K_1-1)^2+\frac 2{c_h}\left ((K_3-1)^2+(K_4-1)^2\right )}\\
 &\ll & K_1^22^{\left ( 1+\frac 2{c_h}\right )\frac 1{c_h-1}\left ((K_1-1)^2+(K_2-1)^2+(K_3-1)^2\right )-(K_1-1)^2+\frac 2{c_h}(K_3-1)^2}\\
 &\ll &K_1^22^{\left (\left (1+\frac 2{c_h}  \right )\frac 3{c_h-1}-1+\frac 2{c_h}\right )(K_1-1)^2}\\
 &\ll &K_1^2 2^{\left (-1+\frac 6{c_h-1}\right )(K_1-1)^2-2K_1}
 \end{eqnarray*}
since $c_h>4$.
Finally
\begin{eqnarray*}\int_1^2|E_8(\alpha,K)|\ud\alpha&\ll & \sum_{K_4\le K_3\le K_2\le K}K^2 2^{\left (-1+\frac 6{c_h-1}\right )(K-1)^2-2K}\\
&\ll & K^5 2^{\left (\frac 6{c_h-1} - 1\right )(K-1)^2-2K},\end{eqnarray*}
as claimed.
\end{proof}

%\todos

\end{document}